\documentclass[12pt]{article}
\usepackage{amsthm, amsmath, amssymb, amsfonts, latexsym}
\date{}

\title{On a multi-phase Stefan problem in the half-line with different boundary conditions at the fixed boundary}
\author{E.\,Yu. Panov \\
St. Petersburg Department of V.\,A.~Steklov Institute of Mathematics \\ of the Russian Academy of Sciences,
St. Petersburg, Russia; \\
Yaroslav-the-Wise Novgorod State University, \\ Veliky Novgorod, Russia}

\theoremstyle{plain}
\newtheorem{theorem}{Theorem}[section]
\newtheorem{lemma}{Lemma}[section]
\newtheorem{proposition}{Proposition}[section]
\newtheorem{corollary}{Corollary}[section]
\theoremstyle{definition}

\newtheorem{remark}{Remark}[section]
\newtheorem{example}{Example}[section]
\numberwithin{equation}{section}

\newcommand{\R}{{\mathbb R}}
\newcommand{\N}{{\mathbb N}}

\newcommand{\Z}{{\mathbb Z}}

\newcommand{\const}{\mathrm{const}}
\begin{document}
\maketitle

\begin{abstract}
We study self-similar solutions of a multi-phase Stefan problem for a heat equation on the half-line $x>0$ with a constant initial data and with Dirichlet, Neumann or Robin boundary condition at the fixed boundary $x=0$. In the case of Dirichlet boundary condition we prove that a nonlinear algebraic system for determination of the free boundaries is gradient one and the corresponding potential is an explicitly written strictly convex and coercive function. Therefore, there exists a unique minimum point of the potential, coordinates of this point determine free boundaries and provide the desired solution. In the case of Neumann boundary condition the study is complicated by the fact that number of phase transitions undergone by a solution (called its type) cannot be directly determined from the boundary data. For each fixed type $n$ the system for determination of the free boundaries is again gradient and the corresponding potential is proved to be strictly convex and coercive, but in some wider non-physical domain. On the base of these properties it is proved that the Dirichlet-to-Neumann map is a strictly increasing continuous function. This allows to establish existence and uniqueness of a solution to Stefan-Neumann problem, and to specify the type of this solution. The same technique is further applied to Stefan-Robin problem with a positive connection coefficient. In the last section we also study some particular ill-posed Stefan-Robin problem with a negative connection coefficient, using again the variational approach developed in the previous sections.
\end{abstract}



\section{Stefan problem with Dirichlet boundary condition}\label{sec1}

In a quarter-plane $\Pi_+=\{ \ (t,x)\in\R^2 \ | \ t,x>0 \ \}$ we consider the multi-phase Stefan problem for the heat equation
\begin{equation}\label{1}
u_t=a_i^2u_{xx}, \quad u_i<u<u_{i+1}, \ i=0,1,\ldots,
\end{equation}
where $u_0$ is an initial temperature while $u_i$, $i\in\N$, are temperatures of phase transitions greater than $u_0$ numerated in increasing order. We assume that the set of these temperatures is infinite and that $\lim\limits_{i\to\infty} u_i=+\infty$. The case of finite number of $u_i$ is simpler, so we omit it.
The coefficients $a_i>0$, $i=0,1,\ldots$, are the diffusivity constants. We will study continuous piecewise smooth solutions $u=u(t,x)$ in $\Pi_+$ satisfying (\ref{1}) in the classical sense in the domains $u_i<u(t,x)<u_{i+1}$, $i=0,1,\ldots$, filled with the phases. On the unknown lines $x=x_i(t)$ of phase transitions where $u=u_i$ the following Stefan condition
\begin{equation}\label{St}
d_ix_i'(t)+k_iu_x(t,x_i(t)+)-k_{i-1}u_x(t,x_i(t)-)=0
\end{equation}
is postulated, where $k_i>0$ is the thermal conductivity of the $i$-th phase, while $d_i\ge 0$ is the Stefan number (the latent specific heat) for the $i$-th phase transition. In (\ref{St}) the unilateral limits $u_x(t,x_i(t)+)$, $u_x(t,x_i(t)-)$ on the line $x=x_i(t)$ are taken from the domain corresponding to the warmer/colder phase, respectively. 

It is known that problem (\ref{1}), (\ref{St}) with $u\ge u_0$ reduces to a degenerate nonlinear diffusion  equation (see \cite{Kam}, \cite[Chapter 5]{LSU})
\begin{equation}\label{diff}
G(u)_t-K(u)_{xx}=0,
\end{equation}
where $K(u)$, $G(u)$ are strictly increasing functions on $[u_0,+\infty)$, linear on each interval $(u_i,u_{i+1})$,
$i=\Z_+\doteq\{0\}\cup\N$, with slopes $K'(u)=k_i$, $G'(u)=k_i/a_i^2$, and such that
\[
K(u_i+)-K(u_i-)=0, \quad G(u_i+)-G(u_i-)=d_i, \ i\in\N.
\]
We will study the initial-boundary value problem with constant initial and Dirichlet boundary data
\begin{equation}\label{2}
u(0,x)=u_0 \ \forall x>0, \quad u(t,0)=u_D \ \forall t>0.
\end{equation}
By the invariance of our problem under the transformation group
$(t,x)\to (\lambda^2 t, \lambda x)$, $\lambda\in\R$, $\lambda>0$, it is natural to seek a self-similar solution of problem (\ref{1}), (\ref{St}), (\ref{2}), which has the form $u(t,x)=u(\xi)$, $\xi=x/\sqrt{t}$. In view of (\ref{2}),
\[u(0)=u_D, \quad u(+\infty)\doteq\lim_{\xi\to+\infty} u(\xi)=u_0<u_D.\]
Thus, it is natural to suppose that the function $u(\xi)$ decreases.

For the heat equation
$u_t=a^2 u_{xx}$ a self-similar solution must satisfy the linear ODE $a^2u''=-\xi u'/2$, the general solution of which is
\[
u=C_1F(\xi/a)+C_2, \ C_1,C_2=\const, \mbox{ where } F(\xi)=\frac{1}{\sqrt{\pi}}\int_0^\xi e^{-s^2/4}ds.
\]
This allows to write our solution in the form
\begin{equation}\label{3}
u(\xi)=\left\{\begin{array}{lr} u_i+\frac{u_{i+1}-u_i}{F(\xi_{i+1}/a_i)-F(\xi_i/a_i)}(F(\xi/a_i)-F(\xi_i/a_i)), &
\xi_{i+1}\le\xi<\xi_i, \\ & i=0,\ldots,n-1; \\
u_n-\frac{u_D-u_n}{F(\xi_n/a_n)}(F(\xi/a_n)-F(\xi_n/a_n)), & 0\le\xi<\xi_n,
\end{array}\right.
\end{equation}
where $\xi_0=+\infty$ and $\displaystyle F(+\infty)=\frac{1}{\sqrt{\pi}}\int_0^{+\infty} e^{-s^2/4}ds=1$.
The value $n\in\{0\}\cup\N$ (i.e., the number of phase transitions) will be called the type of the solution and is determined from the condition $u_n<u_D\le u_{n+1}$.

The parabolas $\xi=\xi_i$, $i=1,\ldots,n$, where $u=u_i$, are free boundaries, which must be determined by conditions (\ref{St}).
In the variable $\xi$ these conditions have the form (cf. \cite[Chapter XI]{CJ}, \cite{Wil}), where we replace $u_{n+1}$ by $u_D$ and agree that $\xi_{n+1}=0$
\begin{equation}\label{4}
d_i\xi_i/2+\frac{k_i(u_{i+1}-u_i)F'(\xi_i/a_i)}{a_i(F(\xi_{i+1}/a_i)-F(\xi_i/a_i))}-
\frac{k_{i-1}(u_i-u_{i-1})F'(\xi_i/a_{i-1})}{a_{i-1}(F(\xi_i/a_{i-1})-F(\xi_{i-1}/a_{i-1}))}=0,
\end{equation}
$i=1,\ldots,n$.

In the case $n=1$ system (\ref{4}) reduces to a single equation, which can be uniquely solved. As a result, we obtain the classical Neumann solution of the Stefan problem.
In the general multi-phase case system (\ref{4}) was analysed by D.~G. Wilson in \cite{Wil}. He suggested
to solve the equations (\ref{4}) sequentially, starting with the first one, which allows to express the variables
$\xi_i$, $i=2,\ldots,n$, as functions of the first variables $\xi_1$. Then the last equation in (\ref{4}) turns to a single equation for determining the unknown $\xi_1$. In \cite{Wil} the author justified (by rather complicated analysis) this method showing that all the obtained  equations can be uniquely resolved, so that the system (\ref{4}) has a unique solution. The methods of \cite{Wil} was further developed in paper \cite{ST} for the case of Neumann boundary conditions.

In the present paper we are going to apply a variational approach to study the nonlinear system (\ref{4}), based on the observation that this system is a gradient one and coincides with the equality
$\nabla E_n(\bar\xi)=0$, where the function (potential)
\begin{align}\label{5}
E_n(\bar\xi)=-\sum_{i=0}^{n-1}  k_i(u_{i+1}-u_i)\ln (F(\xi_i/a_i)-F(\xi_{i+1}/a_i))-\nonumber\\  k_n(u_D-u_n)\ln F(\xi_n/a_n)+\sum_{i=1}^n d_i\xi_i^2/4, \\ \nonumber \bar\xi=(\xi_1,\ldots,\xi_n)\in\Omega_n,
\end{align}
the open convex domain $\Omega_n\subset\R^n$ is given by the inequalities $\xi_1>\cdots>\xi_n>0$.
Observe that $E_n(\bar\xi)\in C^\infty(\Omega)$. Since the function $F(\xi)$ takes values in the interval $(0,1)$, all the terms in expression (\ref{5}) are nonnegative while some of them are strictly positive. Therefore, $E_n(\bar\xi)>0$.

\subsection{Coercivity of the function $E_n$ and existence of a solution}

Let us introduce the sub-level sets \[\Omega(c)=\{ \ \bar\xi\in\Omega_n \ | \ E_n(\bar\xi)\le c \ \}, \quad c>0.\]

\begin{proposition}[coercivity]\label{th1}
The sets $\Omega(c)$ are compact for each $c>0$.
In particular, the function $E_n(\bar\xi)$ reaches its minimal value.
\end{proposition}

\begin{proof}
If $\bar\xi=(\xi_1,\ldots,\xi_n)\in\Omega(c)$ then
\begin{align}\label{co1a}
 -k_i(u_{i+1}-u_i)\ln (F(\xi_i/a_i)-F(\xi_{i+1}/a_i))\le E(\bar\xi)\le c, \quad i=0,\ldots,n-1, \\
 \label{co1b}
-k_n(u_D-u_n)\ln F(\xi_n/a_n)\le E(\bar\xi)\le c.
 \end{align}
It follows from (\ref{co1b}) that $F(\xi_n/a_n)\ge e^{-c/(k_n(u_D-u_n))}$, which implies
the low bound
\[\xi_n\ge r_1=a_nF^{-1}(e^{-c/(k_n(u_D-u_n))}).\]
Similarly, we derive from (\ref{co1a}) with $i=0$ that
\[1-F(\xi_1/a_0)\ge e^{-c/(k_0(u_1-u_0))}\] (notice that $F(\xi_0/a_0)=F(+\infty)=1$). Therefore,
$F(\xi_1/a_0)\le 1-e^{-c/(k_0(u_1-u_0))}$. This implies the upper bound $\xi_1\le r_2=a_0F^{-1}(1-e^{-c/(k_0(u_1-u_0))})$.

Further, it follows from (\ref{co1a}) that for all $i=1,\ldots,n-1$
\begin{equation}\label{6}
F(\xi_i/a_i)-F(\xi_{i+1}/a_i)\ge \alpha_i\doteq\exp(-c/(k_i(u_{i+1}-u_i))>0.
\end{equation}
Since $F'(\xi)=\frac{1}{\sqrt{\pi}}e^{-\xi^2/4}<1$, the function $F(\xi)$ is Lipschitz with constant $1$, and it follows from (\ref{6}) that
\[
(\xi_i-\xi_{i+1})/a_i\ge F(\xi_i/a_i)-F(\xi_{i+1}/a_i)\ge \alpha_i, \quad i=1,\ldots,n-1,
\]
and we obtain the estimates $\xi_i-\xi_{i+1}\ge\delta_i=a_i\alpha_i>0$. Thus, the set $\Omega(c)$ is contained in a compact
\[
K=\{ \ \bar\xi=(\xi_1,\ldots,\xi_n)\in\R^n \ | \ r_2\ge\xi_1\ge\cdots\ge\xi_n\ge r_1, \ \xi_i-\xi_{i+1}\ge\delta_i \ \forall i=1,\ldots,n-1 \ \}.
\]
Since $E_n(\bar\xi)$ is continuous on $K$, the set $\Omega(c)$ is a closed subset of $K$ and therefore is compact. For $c>N\doteq\inf E_n(\bar\xi)$, this set is not empty and the function $E_n(\bar\xi)$ reaches on it a minimal value, which is evidently equal to $N$.
\end{proof}

We have established the existence of minimal value $E_n(\bar\xi_0)=\min E_n(\bar\xi)$. At the point $\bar\xi_0$ the required condition $\nabla E_n(\bar\xi_0)=0$ is satisfied, and $\bar\xi_0$ is a solution of the system (\ref{4}). The coordinates
of $\bar\xi_0$ determine the solution (\ref{3}) of our Stefan problem. Thus, we have established the following existence result.

\begin{theorem}\label{th1a}
There exists a self-similar solution (\ref{3}) of the problem (\ref{1}), (\ref{St}), (\ref{2}).
\end{theorem}

\subsection{Convexity of the function $E_n$ and uniqueness of a solution}

In this section we prove that the function $E_n(\bar\xi)$ is strictly convex. Since a strictly convex function can have at most one critical point (and it is necessarily a global minimum), the system (\ref{4}) has at most one solution, that is, a self-similar solution (\ref{3}) of the problem (\ref{1}), (\ref{St}), (\ref{2}) is unique. We will need the following simple lemma proven in \cite{Pan1} (see also \cite{Pan2}). For the sake of completeness we provide it with the proof.

\begin{lemma}\label{lem1}
The function $P(x,y)=-\ln (F(x)-F(y))$ is strictly convex in the half-plane $x>y$.
\end{lemma}

\begin{proof}
The function $P(x,y)$ is infinitely differentiable in the domain $x>y$. To prove the lemma, we need to establish that the Hessian $D^2 P$ is positive definite at every point. By the direct computation we find
\begin{align*}
\frac{\partial^2}{\partial x^2} P(x,y)=\frac{(F'(x))^2-F''(x)(F(x)-F(y))}{(F(x)-F(y))^2}, \\
\frac{\partial^2}{\partial y^2} P(x,y)=\frac{(F'(y))^2-F''(y)(F(y)-F(x))}{(F(x)-F(y))^2}, \
\frac{\partial^2}{\partial x\partial y} P(x,y)=-\frac{F'(x)F'(y)}{(F(x)-F(y))^2}.
\end{align*}
We have to prove positive definiteness of the matrix $Q=(F(x)-F(y))^2 D^2 P(x,y)$ with the components
\begin{align*}
Q_{11}=(F'(x))^2-F''(x)(F(x)-F(y)), \\ Q_{22}=(F'(y))^2-F''(y)(F(y)-F(x)), \ Q_{12}=Q_{21}=-F'(x)F'(y).
\end{align*}
Since $F'(x)=\frac{1}{\sqrt{\pi}}e^{-x^2/4}$, then $F''(x)=-\frac{x}{2}F'(x)$ and the diagonal elements of this matrix can be written in the form
\begin{align*}
Q_{11}=F'(x)(\frac{x}{2}(F(x)-F(y))+F'(x))= \\ F'(x)(\frac{x}{2}(F(x)-F(y))+(F'(x)-F'(y)))+F'(x)F'(y), \\
Q_{22}=F'(y)(\frac{y}{2}(F(y)-F(x))+(F'(y)-F'(x)))+F'(x)F'(y).
\end{align*}
By Cauchy mean value theorem there exists such a value $z\in (y,x)$ that
\[
\frac{F'(x)-F'(y)}{F(x)-F(y)}=\frac{F''(z)}{F'(z)}=-z/2.
\]
Therefore,
\begin{align*}
Q_{11}=F'(x)(F(x)-F(y))(x-z)/2+F'(x)F'(y), \\ Q_{22}=F'(y)(F(x)-F(y))(z-y)/2+F'(x)F'(y),
\end{align*}
and it follows that $Q=R_1+F'(x)F'(y)R_2$, where $R_1$ is a diagonal matrix with the positive diagonal elements
$F'(x)(F(x)-F(y))(x-z)/2$, $F'(y)(F(x)-F(y))(z-y)/2$ while $R_2=\left(\begin{smallmatrix} 1 & -1 \\ -1 & 1\end{smallmatrix}\right)$. Since $R_1>0$, $R_2\ge 0$, then the matrix $Q>0$, as was to be proved.
\end{proof}

\begin{remark}\label{rem1}
In addition to Lemma~\ref{lem1} we observe that the functions $P(x,0)=-\ln F(x)$, $P(+\infty,x)=-\ln(1-F(x))$
of single variable $x$ are strictly convex on $(0,+\infty)$. In fact, it follows from Lemma~\ref{lem1} in the limit as $y\to 0$ that the function $P(x,0)$ is convex on $(0,+\infty)$, moreover,
\[
(F(x))^2\frac{d^2}{dx^2}P(x,0)=F'(x)(\frac{x}{2}F(x)+F'(x))=\lim_{y\to 0}Q_{11}\ge 0.
\]
Since $F'(x)>0$, we find, in particular, that $\frac{x}{2}F(x)+F'(x)\ge 0$.
If $\frac{d^2}{dx^2}P(x,0)=0$ at some point $x=x_0$ then $0=\frac{x_0}{2}F(x_0)+F'(x_0)$ is the minimum of the nonnegative function $\frac{x}{2}F(x)+F'(x)$. Therefore, its derivative $(\frac{x}{2}F+F')'(x_0)=0$. Since $F''(x)=-\frac{x}{2}F'(x)$, this derivative
\[
(\frac{x}{2}F+F')'(x_0)=F(x_0)/2+\frac{x_0}{2}F'(x_0)+F''(x_0)=F(x_0)/2>0.
\]
But this contradicts our assumption. We conclude that $\frac{d^2}{dx^2}P(x,0)>0$ and the function $P(x,0)$ is strictly convex.

The strong convexity of the function $P(+\infty,x)=-\ln(1-F(x))$ is proved similarly. For the sake of completeness, we
provide the details.
In the limit as $x<y\to+\infty$ we derive from Lemma~\ref{lem1} that the function $P(+\infty,x)=\lim\limits_{y\to+\infty}P(y,x)$ is convex on $\R$ and
\[(1-F(x))^2\frac{d^2}{dx^2}P(+\infty,x)=F'(x)(\frac{x}{2}(F(x)-1)+F'(x))\ge 0.\]
If $\displaystyle\frac{d^2}{dx^2}P(+\infty,x)=0$ at some point $x=x_0\in\R$ then $x_0$ is a minimum point of the nonnegative function
$\frac{x}{2}(F(x)-1)+F'(x)$. Therefore,
\[
0=(\frac{x}{2}(F(x)-1)+F'(x))'(x_0)=(F(x_0)-1)/2+F''(x_0)+F'(x_0)x_0/2=(F(x_0)-1)/2<0.
\]
This contradiction implies that $\displaystyle\frac{d^2}{dx^2}P(+\infty,x)>0$ for all $x\in\R$ and, therefore, the function $P(+\infty,x)$ is strictly convex (even on the whole line $\R$).
\end{remark}

Now we are ready to prove the expected convexity of $E_n(\bar\xi)$.

\begin{proposition}\label{th2}
The function $E_n(\bar\xi)$ is strictly convex on $\Omega_n$.
\end{proposition}

\begin{proof}
We introduce the functions
\[P_i(\bar\xi)=-k_i(u_{i+1}-u_i)\ln (F(\xi_i/a_i)-F(\xi_{i+1}/a_i)), \quad i=0,\ldots,n-1.\]
By Lemma~\ref{lem1} and Remark~\ref{rem1} all these functions are convex.
Since
\[
E_n(\bar\xi)=\sum_{i=0}^{n-1} P_i(\bar\xi)-k_n(u_D-u_n)\ln F(\xi_n/a_n)+\sum_{i=1}^m d_i\xi_i^2/4
\]
and all functions in this sum are convex, it is sufficient to prove strong convexity of the sum
\[
\bar E_n(\bar\xi)=\sum_{i=0}^{n-1} P_i(\bar\xi).
\]
This function is defined in the wider domain consisting of vectors in $\R^n$ with decreasing coordinates (without the positivity requirement). We will show that $\bar E_n(\bar\xi)$ is strictly convex in this domain.

By Lemma~\ref{lem1} and Remark~\ref{rem1} all terms in this sum are convex functions. Therefore, the function $\bar E_n$ is convex as well. To prove its strict convexity, we assume that for some vector $\zeta=(\zeta_1,\ldots,\zeta_n)\in\R^n$.
\begin{equation}\label{deg}
D^2 \bar E_n(\bar\xi)\zeta\cdot\zeta=\sum_{i,j=1}^n \frac{\partial^2 \bar E_n(\bar\xi)}{\partial\xi_i\partial\xi_j}\zeta_i\zeta_j=0.
\end{equation}
Since
\[
0=D^2 \bar E_n(\bar\xi)\zeta\cdot\zeta=\sum_{i=0}^{n-1} D^2 P_i(\bar\xi)\zeta\cdot\zeta
\]
while all the terms are nonnegative, we conclude that
\begin{equation}\label{deg1}
D^2 P_i(\bar\xi)\zeta\cdot\zeta=0, \quad i=0,\ldots,n-1.
\end{equation}
By Lemma~\ref{lem1} for $i=1,\ldots,n-1$ the functions $P_i(\bar\xi)$ are strictly convex as a function of two variables $\xi_i,\xi_{i+1}$ and it follows from (\ref{deg1}) that $\zeta_i=\zeta_{i+1}=0$, $i=1,\ldots,n-1$. Observe that in the case $n=1$ there are no such $i$. In this case we apply (\ref{deg1}) for $i=0$. Taking into account Remark~\ref{rem1}, we find that $P_0(\bar\xi)$ is a strictly convex function of the single variable $\xi_1$, and it follows from (\ref{deg1}) that $\zeta_1=0$. In any case we obtain that the vector $\zeta=0$. Thus, relation (\ref{deg}) can hold only for zero $\zeta$, that is, the matrix $D^2\tilde E_n(\bar\xi)$ is (strictly) positive definite, and the function $\bar E_n(\bar\xi)$ is strictly convex. This completes the proof.
\end{proof}
Propositions~\ref{th1},\ref{th2} imply the main result of this section.

\begin{theorem}\label{th3}
There exists a unique self-similar solution (\ref{3}) of problem (\ref{1}), (\ref{St}), (\ref{2}), and it corresponds to the minimum of strictly convex and coercive function (\ref{5}).
\end{theorem}

\begin{corollary}\label{cor1}
The phase transition parameters $\xi_i$, $i=1,\ldots,n$, corresponding to a solution (\ref{3}) of problem (\ref{1}), (\ref{St}), (\ref{2}) depend continuously on the Dirichlet data $u_D$ in the interval $(u_n,u_{n+1}]$.
\end{corollary}
\begin{proof}
Since the function $E_n(\bar\xi)$ is strictly convex and continuous with respect to the parameter $u_D\in (u_n,u_{n+1}]$, its minimum point $\bar\xi=\bar\xi(u_D)$ is a continuous function of this parameter.
\end{proof}
As follows from Corollary~\ref{cor1} and representation (\ref{3}), the solution $u(\xi)$ of problem (\ref{1}), (\ref{St}), (\ref{2}) depends continuously (in the uniform norm) on $u_D\in (u_n,u_{n+1}]$ for all integer $n\ge 0$. Let us show that it keeps the continuity also at the points $u_n$, $n\in\N$.

\begin{proposition}\label{pro1}
The solution $u(\xi)$ of problem (\ref{1}), (\ref{St}), (\ref{2}) depends continuously (in the uniform norm) on the boundary data $u_D$ in the ray $(u_0,+\infty)$.
\end{proposition}

\begin{proof}
As was mentioned above, the solution $u(\xi)$ is continuous in $u_D$ on each interval $(u_n,u_{n+1}]$. Hence, we only need to prove the right continuity of the solution at points $u_n$, $n\in\N$. Let $u_{Dr}\in (u_n,u_{n+1})$, $r\in\N$, be a decreasing sequence converging as $r\to\infty$ to $u_n$, $u^r(\xi)$ be a solution to problem (\ref{1}), (\ref{St}), (\ref{2}) with the Dirichlet data $u_D=u_{Dr}$, $\bar\xi^r=(\xi^r_1,\ldots,\xi^r_n)\in\Omega_n$ be the corresponding vector of phase transition parameters. By Theorem~\ref{th3} $E^r(\bar\xi^r)=\min E^r(\bar\xi)$, where the potentials
\[E^r(\bar\xi)=-\sum_{i=0}^{n-1} k_i(u_{i+1}-u_i)\ln (F(\xi_i/a_i)-F(\xi_{i+1}/a_i))-\nonumber\\  k_n(u_{Dr}-u_n)\ln F(\xi_n/a_n)+\sum_{i=1}^n d_i\xi_i^2/4\] decrease with respect to $r$. This implies that the sequence $E^r(\bar\xi^r)$ decreases as well and therefore bounded from above: $E^r(\bar\xi^r)\le c$ $\forall r\in\N$ where $c$ is some positive constant. Following the prove of Proposition~\ref{th1}, we find that
\[\bar\xi^r\in K=\{ \ \bar\xi=(\xi_1,\ldots,\xi_n)\in\R^n \ | \ r_2\ge\xi_1\ge\cdots\ge\xi_n\ge 0, \ \xi_i-\xi_{i+1}\ge\delta_i \ \forall i=1,\ldots,n-1 \ \},
\]
where $r_2$ and $\delta_i$ are positive constants independent of $r$, The low bounds $r_1=a_nF^{-1}(e^{-c/(k_n(u_{Dr}-u_n))})\to 0$ as $r\to\infty$ so we have to replace it by $0$. Since the set $K$ is compact, we can extract a subsequence $\bar\xi^r$ (without relabeling) converging to some $\bar\xi^*=(\xi^*_1,\ldots,\xi^*_n)\in K$.
Then the sequence $u^r(\xi)$ uniformly on $[0,+\infty)$ converges as $r\to\infty$ to a function
\[
u^*(\xi)=\left\{\begin{array}{lr} u_i+\frac{u_{i+1}-u_i}{F(\xi^*_{i+1}/a_i)-F(\xi^*_i/a_i)}(F(\xi/a_i)-F(\xi^*_i/a_i)), &
\xi^*_{i+1}\le\xi<\xi^*_i, \\ & i=0,\ldots,n-1; \\
u_n, & 0\le\xi<\xi^*_n.
\end{array}\right.
\]
Passing to the limit as $r\to\infty$ in the Stefan relations (\ref{4}) for the solutions $u^r$, we find that these relations hold at the transition points $\xi=\xi^*_i$, $i=1,\ldots,n-1$, and at the point $\xi=\xi^*_n$ whenever
$\xi^*_n>0$. But in the latter case $\xi^*_n>0$ relation (\ref{4}) turn to the equality
\[
d_n\xi^*_n/2+
\frac{k_{n-1}(u_n-u_{n-1})F'(\xi^*_n/a_{n-1})}{a_{n-1}(F(\xi^*_{n-1}/a_{n-1})-F(\xi^*_n/a_{n-1}))}=0,
\]
which is impossible because its left side is positive. We conclude that $\xi^*_n=0$ and $u^*(\xi)$ is a solution
to problem of (\ref{1}), (\ref{St}), (\ref{2}) of type $n-1$ with the boundary data $u_D=u_n$. In particular, the limit of
$u^r(\xi)$ does not depend on a choice of convergent subsequence. This means that the original sequence converges to $u^*(\xi)$ as $r\to\infty$ uniformly on $[0,+\infty)$. This completes the proof.
\end{proof}

Let us demonstrate that dependence of $u(\xi)$ on the Dirichlet data is monotone.

\begin{proposition}\label{cor2}
Let $u_r=u_r(\xi)$ be solutions (\ref{3}) of problem (\ref{1}), (\ref{St}), (\ref{2}) with Dirichlet data $u_{Dr}$, $r=1,2$. If $u_{D1}>u_{D2}$ then $u_1(\xi)>u_2(\xi)$ for all $\xi>0$.
\end{proposition}

\begin{proof}

Assuming the contrary, we can find such a value $\xi>0$ where $u_1(\xi)\le u_2(\xi)$. By the continuity of $u_1,u_2$ and the condition $u_1(0)>u_2(0)$, there exists the minimal value $\xi_*$ among such $\xi$.
It is clear that $u_1(\xi_*)=u_2(\xi_*)\doteq u_*$. Let $n_r=\max\{i \ | \ u_i<u_{Dr}\}$ be the type of solution $u_r$, $r=1,2$. We can find such an integer $i$ between $0$ and $n_2$ that $u_i\le u_*<u_{i+1}$. If $u_*\not=u_i$ we may add $u_*$ as a new (fictive) phase transition temperature, assigning it  the number $i+1/2$ and the Stefan constant $d_{i+1/2}=0$. The $i$-th phase is thus divided into two identical phases with numbers $i$, $i+1/2$ and equal parameters $a_{i+1/2}=a_i$, $k_{i+1/2}=k_i$. Both the solution $u_r(\xi)$, $r=1,2$, are also solutions of the extended Stefan-Dirichlet problem, since the Stefan condition (\ref{St}) at the fictive phase-transition line $\xi=\xi_{i+1/2}=\xi_*$ reduces to $C^1$-smoothness of $u_r$ at the point $\xi_*$.
Thus, without loss of generality, we may initially suppose that $u_*=u_i$. It is clear that $i\le n_2\le n_1$. Let $\bar\xi=(\xi_1,\ldots,\xi_{n_1})\in\Omega_{n_1}$, $\bar\xi'=(\xi'_1,\ldots,\xi'_{n_2})\in\Omega_{n_2}$ be phase transition parameters corresponding to the solutions $u_1(\xi)$, $u_2(\xi)$.
 We introduce the function
\begin{align*} E_1(\bar\xi_1)=-\sum_{l=0}^{i-1} k_l(u_{l+1}-u_l)\ln (F(\xi_l/a_l)-F(\xi_{l+1}/a_l))+\sum_{l=1}^{i-1} d_l\xi_l^2/4, \\
\mbox{ where } \xi_i=\xi_*, \ \bar\xi_1=(\xi_1,\ldots,\xi_{i-1})\in\Omega_1,
\end{align*}
$\Omega_1$ is an open convex set in the space $\R^{i-1}$ defined by the inequalities $\xi_1>\cdots >\xi_{i-1}>\xi_*$. We agree that $\Omega_1=\emptyset$ if $i=1$. By Theorem~\ref{th3}, the points $\bar\xi$, $\bar\xi'$ are minimal points of the (different) potentials $E_{n_r}$. This implies that the reduced points
$(\xi_1,\ldots,\xi_{i-1})$, $(\xi'_1,\ldots,\xi'_{i-1})$ are the minimum points of the function $E_1(\bar\xi_1)$
whenever $i>1$. Like in Proposition~\ref{th2} it is proved that the function $E_1(\bar\xi_1)$ is strictly convex on $\Omega_1$ and therefore its minimum point is unique. Taking also into account that $\xi_i=\xi'_i=\xi_*$, we find that $\xi_l=\xi'_l$ for all $l=1,\ldots,i$. Remark that in the case $i=1$ this statement is evident. By expression (\ref{3}) we conclude that $u_1(\xi)=u_2(\xi)$ for
$\xi\ge\xi_i$. Therefore, $u_1'(\xi_i+)=u_2'(\xi_i+)$. It now follows from Stefan condition (\ref{St}) that $u_1'(\xi_i-)=u_2'(\xi_i-)\doteq u'_*$. We see that functions $u_1(\xi)$, $u_2(\xi)$ are solutions of the same Cauchy problem for the second order ODE
\[a_iu''=-\xi u'/2, \quad u(\xi_*)=u_*, u'(\xi_*)=u'_*
\]
on a sufficiently small interval $(a,\xi_*]$, $0<a<\xi_*$. Therefore, $u_1(\xi)=u_2(\xi)$ on this interval. But this contradicts to the minimality of $\xi_*$, and completes the proof.
\end{proof}

\begin{remark}\label{rem2} In paper \cite{Pan2} Stefan problem (\ref{1}), (\ref{St}) was studied in the half-plane $t>0$, $x\in\R$ with Riemann initial condition
\begin{equation}
\label{R}
u(0,x)=\left\{\begin{array}{lr} u_+, & x>0, \\ u_-, & x<0. \end{array}\right.
\end{equation}
Solutions of this problem have the same structure as in (\ref{3}) and correspond to a unique minimum point of the function similar to (\ref{5}) with only the difference that the parameters $\xi_i$ are not necessarily positive and can take arbitrary real values. Theorem~\ref{th3} was also proven in paper \cite{Pan3}, where the case $u_1=u_0$ was included.

 In recent paper \cite{Pan4} the Stefan problem with Dirichlet and Neumann conditions at moving boundary $x=\alpha t$ is studied. For finite total number of phase transitions
our results concerning Dirichlet and Neumann problems follows from the results of \cite{Pan4} in the particular case $\alpha=0$. In \cite{Pan4} we study also the exotic case of solutions of infinite type using natural infinite-dimensional extension of the potentials $E_n$.

Remark also that in paper \cite{Pan1} the Stefan-Riemann problem (\ref{1}), (\ref{St}), (\ref{R}) was studied in the case of arbitrary (possibly negative) latent specific heats $d_i$. We found a necessary and sufficient condition for coercivity of the potential $E_n(\bar\xi)$, as well as a stronger sufficient condition of its strict convexity. The similar results can be obtained for the Stefan-Dirichlet problem (\ref{1}), (\ref{St}), (\ref{2}).
\end{remark}

\section{Stefan problem with Neumann boundary condition}

Now we consider Stefan problem (\ref{1}), (\ref{St})
with the constant initial data $u_0$ and with Neumann boundary condition:
\begin{equation}\label{Neu}
u(0,x)=u_0 \ \forall x>0, \quad K(u)_x(t,0)=t^{-1/2} b_N \ \forall t>0,
\end{equation}
where $K(u)$ is the diffusion function in equation (\ref{diff}) and $b_N<0$ is a constant. The specific form of Neumann boundary data is connected with the requirement of invariance of our problem under the scaling transformations $(t,x)\to (\lambda^2 t,\lambda x)$, $\lambda>0$. This allows to concentrate on the study of self-similar solutions $u=u(x/\sqrt{t})$ of the problem (\ref{1}), (\ref{St}), (\ref{Neu}).
For such solutions conditions (\ref{Neu}) reduce to the requirements
\begin{equation}\label{Neu1}
K(u)'(0)=b_N, \quad u(+\infty)=u_0.
\end{equation}
Since $K(u)$ is a strictly increasing function and $b_N<0$, we will assume that the function $u(\xi)$ decreases.
Let $u=u(\xi)$ be a decreasing self-similar solution of (\ref{1}), (\ref{St}), (\ref{Neu}). Then $u(0)>u_0$ and there is an integer $n\ge 0$ such that $u_n<u(0)\le u_{n+1}$. This number $n$ is the type of solution $u$. The Neumann condition for this solution reads $k_nu'(0)=b_N$
(notice that $k_nu_x$ is exactly the heat flow through the boundary point $x=0$).
Unlike Dirichlet boundary data, we cannot find the solution type directly from the boundary data. This greatly complicates the study.
A solution of type $0$ does not contain free boundaries and can be found by the formula
\begin{equation}\label{t0}
u(\xi)=u_0+\frac{a_0}{k_0}b_N\sqrt{\pi}(F(\xi/a_0)-1).
\end{equation}
As is easy to verify, $k_0u'(0)=b_N$, $u(+\infty)=u_0$ and requirement (\ref{Neu1}) is satisfied. By easy computation we find $u(0)=u_0-a_0b_N\sqrt{\pi}/k_0$, therefore, the necessary and sufficient condition for existence of a solution (\ref{t0}) is the
inequality $u_0-a_0b_N\sqrt{\pi}/k_0\le u_1$, which can be written in the form
\begin{equation}\label{con0}
-b_N\le \gamma_1\doteq\frac{k_0(u_1-u_0)}{a_0\sqrt{\pi}}.
\end{equation}
A solution of type $n>0$ has structure similar to (\ref{3})
\begin{align}\label{tki}
u(\xi)=u_i+\frac{u_{i+1}-u_i}{F(\xi_{i+1}/a_i)-F(\xi_i/a_i)}(F(\xi/a_i)-F(\xi_i/a_i)), \quad
\xi_{i+1}\le\xi<\xi_i, \ i=0,\ldots,n-1, \\
\label{tkn}
u(\xi)=u_n+\frac{a_n}{k_n}b_N\sqrt{\pi}(F(\xi/a_n)-F(\xi_n/a_n)), \quad 0\le\xi<\xi_n.
\end{align}
The necessary (but not sufficient, as we will soon realize) condition $u(0)\le u_{n+1}$ of existence of such a solution has the form
\begin{equation}\label{conn}
u(0)=\Phi^{ND}_n(-b_N)\doteq u_n-\frac{a_nb_N\sqrt{\pi}}{k_n} F(\xi_n/a_n)\le u_{n+1}.
\end{equation}
In (\ref{conn}) we introduce the Neumann-to-Dirichlet mapping, which maps a Neumann data $-b_N$ to the Dirichlet data $u(0)$. We underline that the value $\xi_n$ depends on $-b_N$.

Assume that $u(\xi)$ is a solution (\ref{tki}), (\ref{tkn}) of type $n$.
On a phase transition lines $\xi=\xi_i$, the Stefan condition reads
\begin{align}\label{sysi}
d_i\xi_i/2+k_i\frac{(u_{i+1}-u_i)F'(\xi_i/a_i)}{a_i(F(\xi_{i+1}/a_i)-F(\xi_i/a_i))}- \nonumber\\
k_{i-1}\frac{(u_i-u_{i-1})F'(\xi_i/a_{i-1})}{a_{i-1}(F(\xi_i/a_{i-1})-F(\xi_{i-1}/a_{i-1}))}=0, \quad i=1,\ldots,n-1, \\
\label{sysn}
d_n\xi_n/2+b_N\sqrt{\pi} F'(\xi_n/a_n)- k_{n-1}\frac{(u_n-u_{n-1})F'(\xi_n/a_{n-1})}{a_{n-1}(F(\xi_n/a_{n-1})-F(\xi_{n-1}/a_{n-1}))}=0, \quad i=n.
\end{align}
Like in the case of Dirichlet boundary condition, this system turns out to be gradient one, it coincides with the equality $\nabla\tilde E_n=0$, where the function
\begin{align}\label{E}
\tilde E_n(\bar\xi)=-\sum_{i=0}^{n-1}k_i(u_{i+1}-u_i)\ln (F(\xi_i/a_i)-F(\xi_{i+1}/a_i))\nonumber\\
+a_nb_N\sqrt{\pi}F(\xi_n/a_n)+\frac{1}{4}\sum_{i=1}^n d_i\xi_i^2, \quad \bar\xi=(\xi_1,\ldots,\xi_n)\in\Omega_n.
\end{align}
We recall that $\Omega_n$ is an open convex cone in $\R^n$ consisting of vectors with strictly decreasing positive coordinates.
Remark that $F''(s)=-s/2F'(s)<0$ for all $s>0$. Since $b_N<0$, this implies that the term $a_nb_N\sqrt{\pi}F(\xi_n/a_n)$ is a strictly convex function of single variable $\xi_n$ on the interval $[0,+\infty)$. As we demonstrated in the proof of Proposition~\ref{th2}, the function
\[
\bar E_n(\bar\xi)=-\sum_{i=0}^{n-1}k_i(u_{i+1}-u_i)\ln (F(\xi_i/a_i)-F(\xi_{i+1}/a_i))
\]
is strictly convex on a cone
\[\bar\Omega_n=\{ \ \bar\xi=(\xi_1,\ldots,\xi_n)\in\R^n \ | \ \xi_1>\cdots>\xi_n\ge 0 \ \}\supset\Omega_n,\]
consisting of points with strictly decreasing nonnegative coordinates. Since
\[\tilde E_n(\bar\xi)=\bar E_n(\bar\xi)+a_nb_N\sqrt{\pi}F(\xi_n/a_n)+\frac{1}{4}\sum_{i=1}^n d_i\xi_i^2,\] the function $\tilde E_n(\bar\xi)$ is strictly convex on $\bar\Omega_n$ as well.
Let us demonstrate that this function is coercive on $\bar\Omega_n$.

\begin{proposition}\label{th4}
For all $c\in\R$ the set $\bar\Omega_n(c)=\{ \ \bar\xi\in\bar\Omega_n \ | \ \tilde E_n(\bar\xi)\le c \ \}$ is compact.
\end{proposition}

\begin{proof}
Suppose that $\bar\xi\in\bar\Omega_n$, $E_n(\bar\xi)\le c$. Then
\begin{align*}
-\sum_{i=0}^{n-1}k_i(u_{i+1}-u_i)\ln (F(\xi_i/a_i)-F(\xi_{i+1}/a_i))+ \frac{1}{4}\sum_{i=1}^n d_i\xi_i^2= \\
\tilde E(\bar\xi)-a_nb_N\sqrt{\pi}F(\xi_n/a_n)\le c_1\doteq c-a_nb_N\sqrt{\pi}.
\end{align*}
Since all the terms of the left-hand side of this inequality are nonnegative, we obtain the relations
\begin{equation}\label{co2a}
-k_i(u_{i+1}-u_i)\ln (F(\xi_i/a_i)-F(\xi_{i+1}/a_i))\le c_1, \ i=0,\ldots,n-1, \\
\end{equation}
the same as inequalities (\ref{co1a}). As follows from (\ref{co2a}),  the set $\bar\Omega_n(c)=\emptyset$ if $c_1<0$. Therefore, we may (and will) suppose that $c_1\ge 0$. Arguing as in the proof of Proposition~\ref{th1}, we derive from (\ref{co2a}) the bounds
\begin{align*}
\xi_1\le r_2=a_0F^{-1}(1-e^{-c_1/(k_0(u_1-u_0))}), \\
(\xi_i-\xi_{i+1})/a_i\ge \alpha_i=\exp(-c_1/(k_i(u_{i+1}-u_i)))>0, \quad i=1,\ldots,n-1.
\end{align*}
Thus, the set $\bar\Omega_n(c)$ is contained in a compact
\[
K=\{ \ \bar\xi=(\xi_1,\ldots,\xi_n)\in\R^n \ | \ r_2\ge\xi_1\ge\cdots\ge\xi_n\ge 0, \ \xi_i-\xi_{i+1}\ge\alpha_i a_i \ \forall i=1,\ldots,n-1 \ \}.
\]
Since $\tilde E_n(\bar\xi)$ is continuous on $K$, the set $\bar\Omega_n(c)$ is a closed subset of $K$ and therefore is compact. This completes the proof.
\end{proof}

It follows from Proposition~\ref{th4} and the strict convexity of function $\tilde E_n$ that there exists a point
$\bar\xi^n=(\xi_1^n,\ldots,\xi_n^n)\in\bar\Omega_n$ of global minimum of $\tilde E_n$, and it is a unique local minimum of this function. There are two possible cases:

A) $\bar\xi^n\in\Omega_n$, i.e. $\xi_n^n>0$. If, in addition, condition (\ref{conn}) is satisfied then there exists a unique solution (\ref{tki}), (\ref{tkn}) of type $n$
with $\xi_i=\xi_i^n$, $i=1,\ldots,n$;

B) $\bar\xi^n\notin\Omega_n$, i.e. $\xi_n^n=0$. Then a solution of type $n$ does not exist. Let us investigate this case more precisely. The necessary and sufficient conditions for the point $\bar\xi^n=(\xi_1^n,\ldots,\xi_{n-1}^n,0)$ to be a minimum point of $\tilde E_n(\bar\xi)$ are the following
\begin{align}
\label{mc1}
\frac{\partial}{\partial\xi_i} \tilde E_n(\bar\xi^n)=0, \ i=1,\ldots,n-1, \\
\label{mc2}
\frac{\partial}{\partial\xi_n} \tilde E_n(\bar\xi^n)\ge 0,
\end{align}
where condition (\ref{mc1}) appears only if $n>1$. Notice that for such $n$
\begin{align*}
\tilde E_n(\xi_1,\ldots,\xi_{n-1},0)=-\sum_{i=0}^{n-1}k_i(u_{i+1}-u_i)\ln (F(\xi_i/a_i)-F(\xi_{i+1}/a_i))
+\frac{1}{4}\sum_{i=1}^{n-1} d_i\xi_i^2, \\ \xi_n=0, \ (\xi_1,\ldots,\xi_{n-1})\in\Omega_{n-1}.
\end{align*}
We see that $\tilde E_n(\xi_1,\ldots,\xi_{n-1},0)$ coincides with the potential $E_{n-1}(\xi_1,\ldots,\xi_{n-1})$, corresponding to
Stefan-Dirichlet problem (\ref{1}), (\ref{St}), (\ref{2}) with $u_D=u_n$.
Relation (\ref{mc1}) means that $\nabla E_{n-1}(\xi_1^n,\ldots,\xi_{n-1}^n)=0$, that is, $(\xi_1^n,\ldots,\xi_{n-1}^n)\in\Omega_{n-1}$ is a unique minimal point of $E_{n-1}(\xi_1,\ldots,\xi_{n-1})$. According to Theorem~\ref{th3}, the coordinates $\xi_i^n$, $i=1,\ldots,n-1$, coincide with the phase transition parameters $\xi_i$ of the unique solution (\ref{3}) of problem (\ref{1}), (\ref{St}), (\ref{2}) with $u_D=u_n$ (in particular, they do not depend on the Neumann data $b_N$ and on parameters $a_i$, $k_i$, $d_i$ with $i\ge n$). As is easy to calculate, condition (\ref{mc2}) reads
\begin{equation}\label{cB}
\frac{k_{n-1}(u_n-u_{n-1})}{\sqrt{\pi}a_{n-1}F(\xi_{n-1}^n/a_{n-1})}+b_N\ge 0.
\end{equation}
This formula remains valid also for $n=1$, in this case one have to take $\xi_{n-1}^n=\xi_0^1=+\infty$, so that
$F(\xi_{n-1}^n/a_{n-1})=F(+\infty)=1$. Under requirement (\ref{cB}) the case B) is realised so that a solution of type $n$ does not exist.

We introduce the Dirichlet-to-Neumann mapping $\Phi^{DN}_n$, which maps a value $r\in (u_n,u_{n+1}]$ to a value $-k_nu'(0)$, where $u=u^r(\xi)$ is a solution of the Stefan-Dirichlet problem (\ref{1}), (\ref{St}), (\ref{2}) with  $u_D=r$. This mapping is inverse to the Neumann-to-Dirichlet mapping $\Phi^{ND}_n$ introduced in (\ref{conn}) above. It follows from (\ref{3}) that
\begin{equation}\label{dnn}
\Phi^{DN}_n(r)=\frac{k_n(r-u_n)}{a_n\sqrt{\pi}F(\xi_n/a_n)},
\end{equation}
where $\xi_n=\xi_n(r)$ depends on the Dirichlet data $r=u(0)$. Notice that ${F(\xi_0/a_0)=F(+\infty)=1}$ and
\begin{equation}\label{dn0}\Phi^{DN}_0(r)=\frac{k_0(r-u_0)}{a_0\sqrt{\pi}}, \quad u_0<r\le u_1,\end{equation}
is a linear function. By the construction $u^r(\xi)$ is a solution of type $n$ to the Stefan-Neumann problem (\ref{1}), (\ref{St}), (\ref{Neu}) with Neumann data
$b_N=-\Phi^{DN}_n(r)$. Therefore, condition (\ref{cB}) cannot hold, and we claim that
\[\Phi^{DN}_n(r)>\gamma_n\doteq\frac{k_{n-1}(u_n-u_{n-1})}{\sqrt{\pi}a_{n-1}F(\xi_{n-1}^n/a_{n-1})}\]
for all $r\in (u_n,u_{n+1}]$. Remark that $\xi^1_0=+\infty$ and the constant $\gamma_1$ coincides with one introduced earlier in (\ref{con0}).

\begin{lemma}\label{lem2}
The function $\Phi^{DN}_n(r)$ is strictly increasing continuous function on $(u_n,u_{n+1}]$ such that
\[
\Phi^{DN}_n(u_n+)=\lim_{r\to u_n+}\Phi^{DN}_n(r)=\gamma_n, \quad \Phi^{DN}_n(u_{n+1})=\gamma_{n+1}.
\]
In particular, $\gamma_{n+1}>\gamma_n$ for all $n\in\Z_+$. Here we agree that $\gamma_0=0$.
\end{lemma}

\begin{proof}
By Corollary~\ref{cor1} the function $\xi_n=\xi_n(r)$ is continuous. In view of (\ref{dnn}) we see that $\Phi^{DN}_n(r)$ is a continuous function on $(u_n,u_{n+1}]$. If this function is not strictly monotone then we can find values $r_1,r_2$ such that $u_n<r_1<r_2<u_{n+1}$ and that $\Phi^{DN}_n(r_1)=\Phi^{DN}_n(r_2)$. Then the functions $u^{r_1}(\xi)$, $u^{r_2}(\xi)$ are different self-similar solutions of type $n$ to Stefan-Neumann problem (\ref{1}), (\ref{St}), (\ref{Neu}) with the same Neumann data $b_N=-\Phi^{DN}_n(r_1)=-\Phi^{DN}_n(r_2)$, which contradicts to the uniqueness of this solution. Thus, $\Phi^{DN}_n(r)$ is a strictly monotone function. Further,
\[\Phi^{DN}_n(u_{n+1})=\frac{k_n(u_{n+1}-u_n)}{a_n\sqrt{\pi}F(\xi^{n+1}_n/a_n)}=\gamma_{n+1}.\]
We notice that by the strict convexity of function (\ref{E}) its minimum point $\bar\xi^n$ depends continuously on the parameter $s=-b_N$. In  particular, the last coordinate $\xi_n^n=\xi_n^n(s)$ is a continuous function of $s$, and the introduced in (\ref{conn}) Neumann-to-Dirichlet map $\Phi^{ND}_n$ is a continuous function. As follows from (\ref{cB}),   $\xi_n^n(\gamma_n)=0$ while $\xi_n^n(s)>0$ for $s>\gamma_n$. Therefore, for sufficiently small $h>0$ the value $$r_h=\Phi^{ND}_n(\gamma_n+h)\in (u_n,u_{n+1}]$$ and $r_h\to u_n$ as $h\to 0$. Since $\Phi^{DN}_n(r_h)=\gamma_n+h$ and the function $\Phi^{DN}_n(r)$ is monotone, we conclude that $\Phi^{DN}_n(u_n+)=\gamma_n$. In the case $n=0$ the equality ${\Phi^{DN}_0(u_0+)=\gamma_0=0}$ directly follows from expression (\ref{dn0}). Remind that $\Phi^{DN}_n(r)>\gamma_n$ for $r>u_n$. In particular,
\[\gamma_{n+1}=\Phi^{DN}_n(u_{n+1})>\gamma_n,\] and the function $\Phi^{DN}_n(r)$ strictly increases.
\end{proof}

It follows from Lemma~\ref{lem2} that the Dirichlet-to-Neumann map
\begin{equation}\label{mapDN}
\Phi^{DN}(r)=\Phi^{DN}_n(r) \ \mbox{ if } u_n<r\le u_{n+1}, \ n\in\Z_+,
\end{equation}
is a strictly increasing continuous function on $(u_0,+\infty)$ with the image $(0,\gamma_\infty)$, where
$\gamma_\infty=\sup_{n\in\N}\gamma_n\le +\infty$. This, together with Theorem~\ref{th3}, implies
the following main results on correctness of problem (\ref{1}), (\ref{St}), (\ref{Neu}).

\begin{theorem}\label{th5} For any Neumann data $-b_N\in (0,\gamma_\infty)$ there exists a unique solution of the Stefan-Neumann problem
(\ref{1}), (\ref{St}), (\ref{Neu}). The type $n$ of this solution is determined by the condition $\gamma_n<-b_N\le\gamma_{n+1}$.
\end{theorem}

If $u(\xi)$ is a solution of the Stefan-Neumann problem then it has finite type $n$ and, in view of Lemma~\ref{lem2}, $-k_nu'(0)\le\gamma_{n+1}<\gamma_\infty$. Thus, the condition $-b_N<\gamma_\infty$ is necessary for existence of a solution. In the case of finite total number of phase transitions the solution exists for each $-b_N>0$ (cf. \cite{Pan4}) and we may take $\gamma_\infty=+\infty$. But in the case of infinitely many phase transitions it is possible that
$\gamma_\infty$ is finite. Let us confirm this by the following example.

\begin{example}
We consider the Stefan problem with the parameters
$u_i=i$, $a_i\equiv 0$, $k_i=2^{-i}$, $i=\Z_+$; $d_i\equiv 0$, $i\in\N$. We set $\alpha_i=F^{-1}(2^{-i})$ (notice that $\alpha_0=+\infty$), and introduce the continuous piecewise smooth function
\[
v(\xi)=i+2-2^{i+1} F(\xi), \quad \alpha_{i+1}\le \xi<\alpha_i, \ i=0,1,\cdots,
\]
so that $v(\alpha_i)=u_i$.
Since $k_iv'(\alpha_i-0)=k_{i-1}v'(\alpha_i+0)=-2F'(\alpha_i)$ and $d_i=0$, Stefan conditions (\ref{St}) hold on the phase transition lines $\xi=\alpha_i$. Therefore, $v(\xi)$ is a decreasing self-similar solution to the Stefan problem under consideration. Notice that the function $K(v)(\xi)=2-2F(\xi)$, therefore, $0<-K(v)'(\xi)=2F'(\xi)<c\doteq 2/\sqrt{\pi}$ for all $\xi>0$. Now, let $u(\xi)$ be a solution of
Stefan-Dirichlet problem (\ref{1}), (\ref{St}), (\ref{2}) with the described parameters and with $u_D=u_n$. For $0<\xi<\alpha_n$ we have $u(\xi)<u_n$, $v(\xi)>u_n$ and in particular $u(\xi)<v(\xi)$. Let us demonstrate that $K(u)'(\xi)>K(v)'(\xi)$ for all $\xi\in (0,\alpha_n)$ (as follows from the assumption $d_i\equiv 0$, the function $K(u)$ is $C^1$-smooth). Assuming the contrary, we find such $\alpha\in (0,\alpha_n)$ that $K(u)'(\alpha)\le K(v)'(\alpha)$. But in view of \cite[Lemma 2]{Pan4}
the map $u(\alpha)\to -K(u)'(\alpha)$, where $u(\xi)\in C([\alpha,+\infty)$ is a self-similar solution to Stefan problem, is strictly increasing. Hence, the condition $u(\alpha)<v(\alpha)$ implies that $K(u)'(\alpha)>K(v)'(\alpha)$ and we came into contradiction. Thus, $-K(u)'(\xi)\le -K(v)'(\xi)<c$ for all $\xi\in (0,\alpha_n)$. It follows that $\gamma_n=-K(u)'(0+)\le c$.
Since $n\in\N$ is arbitrary, we conclude that $\gamma_\infty=\sup\gamma_n \le c<+\infty$.

\end{example}

\section{Stefan problem with Robin boundary condition}\label{sec3}

Now, we consider the Stefan problem with the Robin boundary condition at the fixed boundary $x=0$:
\begin{equation}\label{Rob}
u(0,x)=u_0; \quad -\sqrt{t}K(u)_x(t,0)+\beta(u(t,0)-u_0)=b_R,
\end{equation}
where $\beta\not=0$ (if $\beta=0$ the problem reduces to already studied Stefan-Neumann problem), $b_R>0$.
The problem (\ref{1}), (\ref{St}), (\ref{Rob}) is again invariant under the scaling transformations $(t,x)\to (\lambda^2 t,\lambda x)$ and we concentrate on self-similar solutions $u(t,x)=u(\xi)$, $\xi=x/\sqrt{t}$. Then, conditions (\ref{Rob})
turn to the following requirements
\begin{equation}\label{Robs}
u(+\infty)=u_0; \quad -K(u)'(0)+\beta(u(0)-u_0)=b_R.
\end{equation}

\subsection{The case of positive connection coefficient}\label{sec3pos}

First, we consider the case when the connection coefficient $\beta>0$. Then the Robin condition in (\ref{Robs}) reduces
to the equality $\Phi^{DR}(u(0))=b_R$, where the function $\Phi^{DR}(r)=\Phi^{DN}(r)+\beta(r-u_0)$ is defined on the ray
$r\in (u_0,+\infty)$. As was demonstrated in the previous section, the Dirichlet-to-Neumann function $\Phi^{DN}(r)$, defined in (\ref{mapDN}), is strictly increasing and continuous with the image $(0,\gamma_\infty)$. Since $\beta>0$, we derive that the function $\Phi^{DR}(r)$ is strictly increasing homeomorphism of $(u_0,+\infty)$ onto $(0,+\infty)$. This observation and Theorem~~\ref{th3} allow to conclude that our Stefan-Robin problem is well-posed:

\begin{theorem}\label{th6} For any Robin data $b_R>0$ there exists a unique solution of problem
(\ref{1}), (\ref{St}), (\ref{Rob}). The type $n$ of this solution is defined from the condition \[\gamma_n+\beta(u_n-u_0)<b_R\le\gamma_{n+1}+\beta(u_{n+1}-u_0).\]
\end{theorem}

\subsection{The case of negative connection coefficient}\label{sec3neg}

In the case $\beta<0$ the monotonicity of the map $\Phi^{DR}(r)$ may be violated and the Stefan-Robin problem is generally ill-posed. Nevertheless, we are going to analyse this problem using the variational approach developed in the previous sections. We will assume that the type $n$ of a solution is fixed. It is convenient to rewrite the Robin condition in the form
\begin{equation}\label{Robn}
k_n u'(0)+\beta(u(0)-u_R)=0,
\end{equation}
where now $\beta>0$. Since a solution $u(\xi)$ decreases, it follows from (\ref{Robn}) that $u(0)>u_R$. We suppose that $u_R>u_n$, then the necessary condition $u(0)>u_n$ is satisfied. By simple calculations we find that a solution of type $n$ has the form
\begin{align}\label{stri}
u(\xi)=u_i+\frac{u_{i+1}-u_i}{F(\xi_{i+1}/a_i)-F(\xi_i/a_i)}(F(\xi/a_i)-F(\xi_i/a_i)), \quad
\xi_{i+1}\le\xi<\xi_i, \ i=0,\ldots,n-1, \\
\label{strn}
u(\xi)=u_n+\frac{u_R-u_n}{F_n-F(\xi_n/a_n)}(F(\xi/a_n)-F(\xi_n/a_n)), \quad 0\le\xi<\xi_n,
\end{align}
where $F_n=k_n/(a_n\beta\sqrt{\pi})>0$. Here, as above, $\xi_0=\infty$, and $\xi_i$, $i=1,\ldots,n$, are phase transition parameters, so that $u(\xi_i)=u_i$. Since $u(\xi)$ decreases, then by (\ref{strn}) we see that $F(\xi_n/a_n)>F_n$. In particular, we have to assume that $F_n<1$. Taking $\xi^*_n=a_nF^{-1}(F_n)>0$, we can write $F_n=F(\xi^*_n/a_n)$.

The Stefan conditions at the lines $\xi=\xi_i$, $i=1,\ldots,n$, for the solution (\ref{stri}), (\ref{strn}) form the algebraic system with the same first $n-1$ equations as in (\ref{sysi}) but with the last equation replaced by the following one
\[
d_n\xi_n/2+\frac{k_n(u_R-u_n)F'(\xi_n/a_n)}{a_n(F_n-F(\xi_n/a_n))}- k_{n-1}\frac{(u_n-u_{n-1})F'(\xi_n/a_{n-1})}{a_{n-1}(F(\xi_n/a_{n-1})-F(\xi_{n-1}/a_{n-1}))}=0.
\]
This system is again gradient, it coincides with the equality $\nabla E=0$, where the function
\begin{align*}
E(\bar\xi)=-\sum_{i=0}^{n-1}k_i(u_{i+1}-u_i)\ln (F(\xi_i/a_i)-F(\xi_{i+1}/a_i))\nonumber\\
-k_n(u_R-u_n)\ln (F(\xi_n/a_n)-F_n)+\frac{1}{4}\sum_{i=1}^n d_i\xi_i^2, \quad \bar\xi=(\xi_1,\ldots,\xi_n)\in\Omega,
\end{align*}
where
\[\Omega=\{ \ \bar\xi\in\R^n \ | \ \xi_1>\cdots>\xi_n>\xi_n^* \ \}.\]
In the same way as for the potential $E_n$ (see Propositions~\ref{th1},~\ref{th2}), we establish that the function $E(\xi)$
is coercive and strictly convex in $\Omega$. Here we also take into account that the term $-k_n(u_R-u_n)\ln (F(\xi_n/a_n)-F_n)=-k_n(u_R-u_n)\ln (F(\xi_n/a_n)-F(\xi^*_n/a_n))$ is strictly convex in $\xi_n>\xi_n^*$ by Lemma~\ref{lem1}. Hence, the function $E(\bar\xi)$  has a unique critical point, the point of its global minimum.
This implies uniqueness of the solution and its existence whenever the requirement $u(0)\le u_{n+1}$ is satisfied. By the property of monotone dependence of a solution on $u(0)$ (see Proposition~\ref{cor2}) this requirement is equivalent to the inequality $\xi_n\le\xi_n^{n+1}$, where $\xi_n$ is the last coordinates of minimum point $\bar\xi$ of the potential $E(\bar\xi)$ while $\xi_n^{n+1}$ is the last coordinate of the point $\bar\xi^{n+1}\in\Omega_n$ (corresponding to the solution of the Stefan-Dirichlet problem with $u(0)=u_{n+1}$). In particular, $\xi^*_n<\xi_n^{n+1}$, which is equivalent to the inequality
\begin{equation}\label{nc}
k_n/(a_n\beta\sqrt{\pi})=F_n<F(\xi_n^{n+1}/a_n),
\end{equation}
being necessary for the existence. Let us show that for small $u_R-u_n$ this inequality is also sufficient.

\begin{theorem}\label{th7}
Assume that condition (\ref{nc}) is satisfied. Then there exists a positive $\varepsilon$ such that for all $u_R\in (u_n,u_n+\varepsilon)$ the problem (\ref{1}), (\ref{St}), (\ref{Robn}) has a unique solution of type~$n$.
\end{theorem}

\begin{proof}
We introduce the function $f(r)=\Phi^{DN}(r)-\beta(r-u_R)$, $r\in [u_,u_{n+1}]$, where $\Phi^{DN}$ is the Dirichlet-to-Neumann function (\ref{mapDN}). Obviously, Robin condition (\ref{Robn}) reduces to the equality
$f(u(0))=0$. Notice that $f(u_n)=\gamma_n+\beta(u_R-u_n)>0$,
\[f(u_{n+1})=\gamma_{n+1}+\beta(u_R-u_{n+1})=\frac{k_n(u_{n+1}-u_n)}{a_n\sqrt{\pi}F(\xi^{n+1}_n/a_n)}-\beta(u_{n+1}-u_n)+\beta(u_R-u_n).\]
Since $\beta=\frac{k_n}{a_n\sqrt{\pi}F_n}$, this expression can be written as
\[\frac{k_n(u_{n+1}-u_n)}{a_n\sqrt{\pi}}[1/F(\xi^{n+1}_n/a_n)-1/F_n]+\beta(u_R-u_n).\]
By (\ref{nc}) $F_n<F(\xi^{n+1}_n/a_n)$ and we conclude that $f(u_{n+1})<0$ for
\[u_R-u_n<\varepsilon=\frac{k_n(u_{n+1}-u_n)}{a_n\beta\sqrt{\pi}}[1/F_n - 1/F(\xi^{n+1}_n/a_n)].\]
Since the continuous function $f(r)$ changes sign on the segment $[u_n,u_{n+1}]$ there exists
$u(0)\in (u_n,u_{n+1})$ such that $f(u(0))=0$. The corresponding solution of the Stefan-Dirichlet problem will be a desired solution of problem (\ref{1}), (\ref{St}), (\ref{Robn}) of type $n$. As we have already demonstrated, such a solution is unique. The proof is complete.
\end{proof}

\end{document}